\documentclass[pdflatex,sn-mathphys-num]{sn-jnl}

\usepackage{amsmath,amssymb,amsfonts,mathtools}
\usepackage{mathrsfs}
\usepackage{booktabs}
\usepackage{xcolor}
\usepackage{hyperref}

\theoremstyle{thmstyleone}
\newtheorem{theorem}{Theorem}[section]
\newtheorem{proposition}[theorem]{Proposition}
\newtheorem{lemma}[theorem]{Lemma}
\newtheorem{corollary}[theorem]{Corollary}

\theoremstyle{thmstyletwo}
\newtheorem{remark}[theorem]{Remark}

\theoremstyle{thmstylethree}
\newtheorem{definition}[theorem]{Definition}

\DeclareMathOperator{\Law}{Law}
\DeclareMathOperator{\supp}{supp}
\DeclareMathOperator{\Aut}{Aut}

\newcommand{\N}{\mathbb N}
\newcommand{\R}{\mathbb R}

\newcommand{\J}{\mathcal J}

\newcommand{\Kcal}{\mathcal K}

\newcommand{\Dcal}{\mathcal D}
\newcommand{\Xbf}{\mathbf X}
\newcommand{\Ybf}{\mathbf Y}
\newcommand{\Zbf}{\mathbf Z}
\newcommand{\Id}{\mathrm{id}}
\newcommand{\dd}{\,\mathrm d}

\begin{document}

\title[Observing joinings]{Observing Joinings: A Distance-Array Characterization of Furstenberg Disjointness}

\author*[1,2]{\fnm{Ao} \sur{Xu}}\email{xuao24@mails.jlu.edu.cn}

\affil[1]{\orgdiv{School of Artificial Intelligence}, \orgname{Jilin University}, \orgaddress{\street{No. 2699 Qianjin Street}, \city{Changchun}, \postcode{130012}, \country{China}}}

\affil[2]{\orgname{Zhongguancun Academy}, \orgaddress{\street{Daniufang 2nd Ring Road, Haidian District}, \city{Beijing}, \postcode{100094}, \country{China}}}

\abstract{Joinings are fundamental global objects in ergodic theory, yet in compact metric models one naturally observes only finite orbit-distance patterns.  We bridge this gap by introducing multi-particle distance arrays, which sample finite orbit segments and record their joint metric evolution.  In the anchored fixed-model setting, this framework yields a purely finite-observable characterization of Furstenberg disjointness: two systems are disjoint if and only if all their anchored multi-orbit distance-array projections are independent.  The structural engine behind this criterion is a marked and colored version of the Gromov--Vershik reconstruction principle for exchangeable arrays; unanchored arrays reconstruct the intrinsic twin-free quotient, while anchors recover the actual joining in a fixed model.  To quantify this independence, we introduce Wasserstein dependence coefficients, establishing an all-order zero criterion for disjointness, and show that weak neighborhoods of the product joining always admit finite distance-array certificates.  Examples from compact rotations, Bernoulli and reversible Markov shifts, common factors, Kronecker factors, and weak mixing demonstrate the strict necessity of the multi-particle level and the broad scope of this approach.}



\maketitle

\section{Introduction}

Joinings are among the central objects of ergodic theory, beginning with
Furstenberg's original formulation of disjointness
\cite{furstenberg1967disjointness,furstenberg1981recurrence,glasner2003ergodic}.
If
\[
  \Xbf=(X,\mathcal B_X,\mu,T),\qquad
  \Ybf=(Y,\mathcal B_Y,\nu,S)
\]
are probability-preserving systems, a joining is a probability measure on
$X\times Y$ with marginals $\mu,\nu$ that is invariant under $T\times S$.
The systems are disjoint, in the sense of Furstenberg, when the only joining
is the product measure $\mu\otimes\nu$.  The force of this definition is that
it is global.  It asks whether two systems admit any stationary coupling beyond
independence, not whether a particular finite statistic happens to decorrelate.
This global viewpoint is one reason joining and self-joining methods have
become central in the structure theory of measure-preserving systems, from
classical disjointness theory to rigidity and characteristic-factor phenomena
\cite{furstenberg1977diagonal,rudolph1979minimal,ratner1991raghunathan,host2005nonconventional}.

This global nature creates the basic tension of the paper.  In a compact
metric model, the most natural finite observations are not arbitrary Borel
sets; they are orbit segments and metric relations among orbit points.  A
joining is an invariant measure on a product space, but the visible data in a
metric model consist of finite distance patterns.  The question we address is:
to what extent can joining theory, and in particular disjointness, be recovered
from such finite metric shadows?  This is a measure-theoretic reconstruction
problem from stationary metric observations.  It is philosophically adjacent
to finite-observation problems in smooth dynamics, such as delay-coordinate
embedding theorems \cite{takens1981detecting}, but the object reconstructed
here is not a topological attractor; it is a joining, hence a stationary
coupling modulo null sets.

There is an immediate obstruction.  The distance matrix of one orbit segment
may contain no phase information.  For an ergodic rotation on a compact
abelian group with a translation-invariant metric, the single-orbit distance
matrix is deterministic; it cannot distinguish the product joining from graph
joinings, or even distinguish different graph joinings.  Thus the observable
unit cannot be a single orbit.  One has to compare several independently
sampled orbit segments at once.  The passage from one orbit to many particles
is the essential step.

Given a joining $\lambda$ of compact metric measure-preserving systems
$(X,d_X,\mu,T)$ and $(Y,d_Y,\nu,S)$, we sample
$(x_i,y_i)_{i=1}^n$ independently from $\lambda$ and record all distances
between the length-$m$ orbit segments in each coordinate:
\[
  \Dcal^X_{n,m}
  =
  \bigl(d_X(T^a x_i,T^b x_j)\bigr)_{
    1\leq i,j\leq n,\ 0\leq a,b<m},
\]
and
\[
  \Dcal^Y_{n,m}
  =
  \bigl(d_Y(S^a y_i,S^b y_j)\bigr)_{
    1\leq i,j\leq n,\ 0\leq a,b<m}.
\]
The law
\[
  \Phi_{n,m}(\lambda)
  =
  \Law_{\lambda^{\otimes n}}(\Dcal^X_{n,m},\Dcal^Y_{n,m})
\]
is the finite distance-array projection of the joining.  These projections are
unanchored: they remember only the intrinsic metric relations among the sampled
orbit points, not the names of the points in a fixed model.

The reason these arrays are faithful enough is conceptual rather than
accidental.  They are finite-dimensional distributions of an exchangeable
array whose edge colors are the orbit-distance kernels
\[
  K_X^{a,b}((x,y),(x',y'))=d_X(T^a x,T^b x'),\qquad
  K_Y^{a,b}((x,y),(x',y'))=d_Y(S^a y,S^b y'),
\]
and whose vertex marks are the same-particle orbit distances.  Thus the
problem sits at the intersection of two reconstruction theories: the
Gromov--Vershik reconstruction of metric measure spaces
\cite{gromov1999metric,vershik2004random,sturm2006geometry,greven2009convergence,memoli2011gromov}
and the metric-dynamical perspective on matrix distributions
\cite{vershik2023dynamics}, together with the
Aldous--Hoover--Kallenberg theory of exchangeable arrays
\cite{aldous1981representations,hoover1979relations,kallenberg2005probabilistic}.
In modern combinatorics, the same pure-function reconstruction principle is
one of the engines behind graph limit theory and graphons
\cite{lovasz2006limits,borgs2008moments,diaconis2008graph,lovasz2012large}; our arrays are continuous-valued,
directed, marked and countably colored analogues forced by orbit geometry.
We use the corresponding measurable pure-function reconstruction theorem in a
marked and colored form.  The unavoidable price is the familiar one:
unanchored data reconstruct only the twin-free intrinsic quotient, unless a
purity hypothesis is imposed.

\subsection*{Main results}
The first main result is an intrinsic reconstruction theorem.  Equality of all
finite unanchored distance-array laws for two joinings implies that the
associated marked colored orbit-distance kernel spaces have isomorphic
twin-free quotients.  If the joinings are distance-array pure, this gives a
kernel-preserving isomorphism of the joining spaces.  In support-reduced
compact metric models, such an isomorphism splits into coordinate
metric-dynamical automorphisms.  Thus the unanchored theory reconstructs the
joining precisely up to the symmetries that unanchored metric data cannot see.

The second main result is the fixed-model version.  We add a countable dense
family of anchors in each support and record distances from orbit points to
these anchors.  The resulting anchored hierarchy is faithful: if two joinings
have the same anchored distance-array laws at all orders, then they are equal
as measures on the fixed product model.

The third main result is the disjointness criterion.  Furstenberg disjointness
is equivalent, in the anchored fixed-model sense, to finite pattern
independence:
\[
  \J(T,S)=\{\mu\otimes\nu\}
  \quad\Longleftrightarrow\quad
  \widetilde\Phi_{n,m,R}(\J(T,S))
  =
  \{\widetilde\Phi_{n,m,R}(\mu\otimes\nu)\}
  \quad\forall n,m,R.
\]
In words, every stationary coupling has only product shadows in every finite
multi-orbit distance-array projection.  The unanchored formulation gives the
corresponding intrinsic quotient-level statement.

We also introduce Wasserstein dependence coefficients obtained by taking the
largest Wasserstein distance, over all joinings, between a finite
distance-array projection and its product counterpart.  Their anchored
all-order vanishing is equivalent to disjointness.  This quantitative
geometric viewpoint parallels the use of optimal transport and
Kantorovich--Rubinstein duality in probability
\cite{kantorovich1958space,villani2009optimal}, and
also resonates with Ornstein's $\bar d$-metric approach to isomorphism theory,
where orbit names are compared through optimal stationary couplings
\cite{ornstein1970bernoulli,ornstein1974ergodic}.  Recent work on stationary
optimal transport likewise studies optimal couplings under invariance
constraints, equivalently optimal joinings for a prescribed cost
\cite{oconnor2024stationary}.  Departing from these classical setups, the present paper uses this circle of ideas in
a different direction: the transport is not placed directly on symbolic names
or on a fixed cost over the phase space, but on finite observable shadows of
joining-induced distance arrays, and the relevant extremum is over all
joinings as a test for disjointness.  In addition, a finite certificate theorem
shows that weak neighborhoods of the product joining are controlled by
finitely many anchored distance-array statistics.  The examples at the end of
the paper include compact rotations, Bernoulli and reversible Markov shifts,
common factors, Kronecker factors, and the classical characterization of weak
mixing by disjointness from Kronecker systems.

\section{Metric systems, joinings, and distance arrays}

\begin{definition}
A \emph{compact metric measure-preserving system} is a quadruple
\[
  \Xbf=(X,d_X,\mu,T),
\]
where $(X,d_X)$ is compact metric, $\mu$ is a Borel probability measure, and
$T:X\to X$ is a Borel map satisfying $T_\#\mu=\mu$.  It is \emph{invertible}
if $T$ has a measurable inverse modulo $\mu$ and $T^{-1}_\#\mu=\mu$.
\end{definition}

Throughout the paper all probability spaces are standard Borel spaces
completed with respect to the relevant probability measure.  Compactness is
used to avoid moment assumptions; all distance arrays then take values in a
compact cube.  Unless explicitly stated otherwise, factor maps, isomorphisms,
and commutation relations between measure-preserving transformations are
understood modulo null sets.  We use standard Borel facts such as
disintegration, the Doob--Dynkin factorization lemma, and the realization of
measure-algebra isomorphisms modulo null sets; see, for example,
\cite[Ch.~1]{kallenberg2005probabilistic}.

\begin{remark}[Support convention]
\label{rem:support-convention}
When convenient we replace a compact metric model $(X,d,\mu,T)$ by its
measure support $\supp\mu$.  If $T$ is continuous and $T_\#\mu=\mu$, then
\[
  T(\supp\mu)\subset\supp\mu .
\]
Indeed, if $x\in\supp\mu$ and $U$ is an open neighborhood of $Tx$, then
$T^{-1}U$ is an open neighborhood of $x$, so
$\mu(U)=\mu(T^{-1}U)>0$.  Thus $Tx\in\supp\mu$.  After this harmless
restriction the measure has full support; we call such a model
\emph{support-reduced}.  Anchors are always understood to be dense in the
relevant supports; in a support-reduced model this simply means dense in the
ambient compact spaces.
\end{remark}

\begin{lemma}[Distance coordinates generate the support Borel structure]
\label{lem:distance-coordinates-generate}
Let $(X,d)$ be compact, let $\mu$ be a Borel probability measure with full
support, and let $(a_r)_{r\ge1}$ be dense in $X$.  Then
\[
  \iota:X\to\R^\N,\qquad \iota(x)=(d(x,a_r))_{r\ge1},
\]
is continuous and injective.  Consequently it is a Borel isomorphism from $X$
onto its image, and the functions $x\mapsto d(x,a_r)$ generate
$\mathcal B_X$.  Equivalently, all equalities of generated
$\sigma$-algebras below may be read in the completed measure algebra.
\end{lemma}

\begin{proof}
Continuity is immediate.  If $x\ne x'$, choose $r_k$ with $a_{r_k}\to x$.
Then
\[
  d(x',a_{r_k})\to d(x',x)>0,
  \qquad
  d(x,a_{r_k})\to0,
\]
so $\iota(x)\ne\iota(x')$.  A continuous injective map from a compact space
to a Hausdorff space is a homeomorphism onto its image.  Hence the coordinate
functions of $\iota$ generate the Borel structure.
\end{proof}

\begin{definition}
Let $\Xbf=(X,d_X,\mu,T)$ and $\Ybf=(Y,d_Y,\nu,S)$ be compact metric
measure-preserving systems.  A \emph{joining} of $\Xbf$ and $\Ybf$ is a Borel
probability measure $\lambda$ on $X\times Y$ such that
\[
  (\pi_X)_\#\lambda=\mu,\qquad
  (\pi_Y)_\#\lambda=\nu,\qquad
  (T\times S)_\#\lambda=\lambda .
\]
The convex set of joinings is denoted by $\J(T,S)$.  When $T$ and $S$ are
continuous, $\J(T,S)$ is weakly compact.  The systems are \emph{disjoint} if
$\J(T,S)=\{\mu\otimes\nu\}$.
\end{definition}

\begin{definition}
For $\lambda\in\J(T,S)$, integers $n,m\geq1$, and
$z_i=(x_i,y_i)\in X\times Y$, define
\[
  \Dcal^X_{n,m}(z_1,\ldots,z_n)
  =
  \bigl(d_X(T^a x_i,T^b x_j)\bigr)_{
  1\leq i,j\leq n,\ 0\leq a,b<m},
\]
\[
  \Dcal^Y_{n,m}(z_1,\ldots,z_n)
  =
  \bigl(d_Y(S^a y_i,S^b y_j)\bigr)_{
  1\leq i,j\leq n,\ 0\leq a,b<m}.
\]
The \emph{unanchored distance-array projection} of $\lambda$ is
\[
  \Phi_{n,m}(\lambda)
  =
  \Law_{\lambda^{\otimes n}}
  \bigl(\Dcal^X_{n,m},\Dcal^Y_{n,m}\bigr).
\]
\end{definition}

The product joining has projection
\[
  \Phi_{n,m}(\mu\otimes\nu)
  =
  \Law_{\mu^{\otimes n}\otimes\nu^{\otimes n}}
  \bigl(\Dcal^X_{n,m},\Dcal^Y_{n,m}\bigr).
\]
Equivalently, if $\rho^X_{n,m}$ and $\rho^Y_{n,m}$ denote the marginal
finite orbit-distance laws generated by $\mu^{\otimes n}$ and
$\nu^{\otimes n}$, then
\[
  \Phi_{n,m}(\mu\otimes\nu)=\rho^X_{n,m}\otimes\rho^Y_{n,m}.
\]
The shared index set creates no dependence under the product joining because
the entire $X$-sample and the entire $Y$-sample are independent.

\begin{proposition}[Continuity and compactness]
\label{prop:compact}
Assume $T$ and $S$ are continuous.  Then, for each $n,m$, the map
\[
  \lambda\mapsto \Phi_{n,m}(\lambda)
\]
from $\J(T,S)$ with the weak topology to the space of probability measures on
the finite array cube with the weak topology is continuous.  In particular
$\Phi_{n,m}(\J(T,S))$ is compact.
\end{proposition}

\begin{proof}
The map
$(z_1,\ldots,z_n)\mapsto(\Dcal^X_{n,m},\Dcal^Y_{n,m})$ is continuous on
$(X\times Y)^n$ because the metrics and the finitely many iterates of $T,S$
appearing in the definition are continuous.  If $\lambda_k\to\lambda$ weakly
on $X\times Y$, then $\lambda_k^{\otimes n}\to\lambda^{\otimes n}$ weakly.
Pushing forward by the continuous array map gives
$\Phi_{n,m}(\lambda_k)\to\Phi_{n,m}(\lambda)$.  The joining set is closed in
the compact set of couplings of $\mu,\nu$: the marginal constraints are closed,
and the invariance constraint is closed because
$\lambda\mapsto (T\times S)_\#\lambda$ is weakly continuous when $T\times S$
is continuous.  Hence $\J(T,S)$ is compact, and the image of a compact set
under a continuous map is compact.
\end{proof}

\section{Why single-orbit arrays are not enough}

The multi-particle definition is forced by a simple obstruction.  A single
orbit segment can be blind to the phase information carried by a joining.

\begin{proposition}[Failure of single-orbit distance arrays]
\label{prop:single-fails}
Let $G$ be a compact abelian group with Haar probability $m_G$, let $d$ be a
translation-invariant metric on $G$, and let $T_g x=x+g$ be an ergodic
rotation.  Then for every self-joining $\lambda$ of this rotation and every
$m\geq1$, the pair of single-orbit distance matrices
\[
 \left(
   (d(T_g^a x,T_g^b x))_{0\leq a,b<m},
   (d(T_g^a y,T_g^b y))_{0\leq a,b<m}
 \right)
\]
has the same deterministic law under $(x,y)\sim\lambda$.  In particular,
single-orbit distance arrays do not distinguish the product joining from any
graph joining
\[
  \lambda_h=(x\mapsto(x,x+h))_\#m_G .
\]
\end{proposition}

\begin{proof}
For every $x\in G$ and $a,b\geq0$,
\[
 d(T_g^a x,T_g^b x)=d(x+ag,x+bg)=d(0,(b-a)g)
\]
by translation invariance.  The same computation applied to the second
coordinate gives, for every $y\in G$,
\[
 d(T_g^a y,T_g^b y)=d(0,(b-a)g)
\]
Hence the two single-orbit matrices are deterministic for
every pair $(x,y)$, independently of the self-joining used to sample that pair.
\end{proof}

\begin{remark}
The joinings $\lambda_h$ need not be equal, and they are generally non-product.
Thus a single orbit-distance observable can collapse the whole self-joining
simplex to one deterministic pattern.  Multi-particle arrays restore the
missing information because they compare the relative positions of
independently sampled orbit segments.
\end{remark}

\section{Marked colored kernel reconstruction}

We now isolate the reconstruction principle behind the unanchored theory.  The
following language is convenient because the distance-array projections of a
joining are exactly finite laws of a countably marked, countably colored
kernel array.

\begin{definition}
A \emph{countably marked colored kernel probability space} is a tuple
\[
  \Kcal=(\Omega,\mathcal F,\rho,(m_\alpha)_{\alpha\in A_0},(k_c)_{c\in C}),
\]
where $(\Omega,\mathcal F,\rho)$ is a standard probability space,
$A_0$ and $C$ are countable, each mark $m_\alpha:\Omega\to\R$ is bounded and
measurable, and each kernel $k_c:\Omega\times\Omega\to\R$ is bounded and
measurable.
Two such spaces are \emph{kernel-isomorphic} if there is a measure-space
isomorphism $U:\Omega\to\Omega'$ such that for every $\alpha\in A_0$,
\[
  m_\alpha(z)=m'_\alpha(Uz)
\]
for $\rho$-almost every $z$, and for every $c\in C$,
\[
  k_c(z,z')=k'_c(Uz,Uz')
\]
for $\rho\otimes\rho$-almost every $(z,z')$.
\end{definition}

Given $\Kcal$ and an i.i.d. sequence $(Z_i)_{i\geq1}$ with law $\rho$, define
the infinite marked colored array
\[
  \mathbf A_\Kcal=
  \bigl((m_\alpha(Z_i))_{\alpha\in A_0,\ i\ge1},
        (k_c(Z_i,Z_j))_{c\in C,\ i\ne j}\bigr).
\]
Its law is denoted by $\mathfrak A(\Kcal)$.

\begin{definition}
Two points $z,z'\in\Omega$ are \emph{twins} if for every $\alpha\in A_0$ and
every $c\in C$,
\[
  m_\alpha(z)=m_\alpha(z'),
\]
and
\[
  k_c(z,w)=k_c(z',w)
  \quad\text{and}\quad
  k_c(w,z)=k_c(w,z')
\]
for $\rho$-almost every $w$.  The space $\Kcal$ is
\emph{row-column-mark pure} if, after discarding a null set, no two distinct
points are twins.
\end{definition}

\subsection{Measurable kernel reconstruction}
\label{subsec:kernel-reconstruction-background}

We now record the measurable reconstruction result used in the sequel and
explain why it applies exactly to the marked colored kernels above.  The relevant
object in Vershik's terminology is the \emph{matrix distribution} of a
measurable function of two variables: if
$f:(\Omega,\rho)^2\to A$ is a measurable map into a standard Borel space and
$(Z_i)_{i\ge1}$ is i.i.d. with law $\rho$, then the off-diagonal law of
$(f(Z_i,Z_j))_{i\ne j}$ is a jointly exchangeable probability measure on
$A^{\{(i,j):i\ne j\}}$.
The precise external input is Vershik's matrix-distribution reconstruction
theorem for pure functions.  In the later corrected formulation
\cite[Theorem~5]{vershik2012classification}, if $D=D_f$ is the matrix
distribution of a totally pure measurable function, then the function
canonically constructed from a $D$-typical matrix is equivalent to $f$ and has
matrix distribution $D$; moreover the simple measures are exactly the matrix
distributions of totally pure measurable functions.  The original paper
\cite{vershik2002classification} introduced the matrix-distribution formalism
and the classification problem.  We use this theorem only through the
row-column indistinguishability quotient described below.  Thus the
reconstruction below is reconstruction at the level of the canonical
twin-free representative, unless purity is imposed separately.
This result is closely related to the Aldous--Hoover--Kallenberg theory of
jointly exchangeable arrays \cite{aldous1981representations,hoover1979relations,kallenberg2005probabilistic}.

In the graphon literature the same point appears as the uniqueness of
twin-free graphon representatives: the distribution of the random infinite
graph sampled from a graphon determines the graphon up to a measure-preserving
isomorphism after twins are identified
\cite{borgs2008moments,lovasz2012large}.  We use the measurable-function
form rather than the graphon form because our kernels are neither binary nor
symmetric, because we need countably many colors, and because the diagonal
blocks of our distance arrays are one-point marks.  These extensions are
formal: countably many bounded real marks and kernels can be encoded in a
single product-valued kernel, and non-symmetric kernels are handled by using
both row and column twins.

\begin{theorem}[Vershik pure-function reconstruction, form used here]
\label{thm:vershik-pure-function}
Let $(\Omega,\mathcal F,\rho)$ be a standard probability space, let $A$ be a
standard Borel space, and let $F:\Omega\times\Omega\to A$ be measurable.  Let
$(Z_i)_{i\ge1}$ be i.i.d. with law $\rho$, and let
\[
  \mathfrak M(F)=\Law\bigl((F(Z_i,Z_j))_{i\ne j}\bigr)
\]
be the off-diagonal matrix distribution.  Define two points $z,z'$ to be
$F$-twins if
\[
  F(z,w)=F(z',w)
  \quad\text{and}\quad
  F(w,z)=F(w,z')
\]
for $\rho$-almost every $w$.  Vershik's theorem asserts that
$\mathfrak M(F)$ determines the pure representative of $F$, obtained by
quotienting by this row-column indistinguishability relation.  Equivalently,
there is a twin-free quotient $(\Omega_F,\rho_F,F^{\rm tw})$, unique up to
null sets and measure-space isomorphism, with the same off-diagonal matrix
distribution; and if $F$ and $F'$ have the same off-diagonal matrix
distribution, then their twin-free quotients are isomorphic.  That is, there
is a measure-space isomorphism $U:\Omega_F\to\Omega_{F'}$ such that
\[
  F^{\rm tw}(z,z')=(F')^{\rm tw}(Uz,Uz')
\]
for $\rho_F\otimes\rho_F$-almost every $(z,z')$.  In particular, if both
$F$ and $F'$ are twin-free, equality of their matrix distributions implies
isomorphism of the original functions.
\end{theorem}

\begin{remark}[Applicability to the present kernels]
\label{rem:vershik-applicability}
Theorem~\ref{thm:vershik-pure-function} is the two-variable case of Vershik's
classification of measurable functions by their matrix distributions
\cite{vershik2002classification,vershik2012classification}.  We use it only in
the following reduced form.  The target $A$ may be any standard Borel space;
in our application it is compact metrizable.  Countably many real-valued marks
and colors are encoded in one product-valued map
\[
  F_\Kcal(z,z')=(m(z),m(z'),k(z,z')),
\]
where $m(z)$ is the vector of vertex marks and $k(z,z')$ is the vector of edge
colors.  Non-symmetric kernels are covered because the twin relation uses both
rows and columns.  The diagonal entries of our distance arrays are not treated
as values of $F$ on the diagonal; they are included as one-point marks in
$m(z)$.  Thus the theorem applies exactly to the marked colored kernel spaces
used below.  The verification of this reduction is included in the proof of
Theorem~\ref{thm:kernel-reconstruction}.
\end{remark}

\begin{theorem}[Measurable marked colored-kernel reconstruction]
\label{thm:kernel-reconstruction}
Let
\[
  \Kcal=(\Omega,\mathcal F,\rho,(m_\alpha)_{\alpha\in A_0},(k_c)_{c\in C})
\]
and
\[
  \Kcal'=(\Omega',\mathcal F',\rho',(m'_\alpha)_{\alpha\in A_0},(k'_c)_{c\in C})
\]
be countably marked colored kernel probability spaces with bounded marks and
kernels.  There is a natural, modulo null sets, row-column-mark pure marked colored kernel
space $\Kcal^{\rm tw}$, called the
\emph{twin-free quotient} of $\Kcal$, with the following properties:
\begin{enumerate}
\item $\mathfrak A(\Kcal^{\rm tw})=\mathfrak A(\Kcal)$;
\item if $\Kcal$ is row-column-mark pure, then $\Kcal^{\rm tw}$ is
kernel-isomorphic to $\Kcal$;
\item if
\[
  \mathfrak A(\Kcal)=\mathfrak A(\Kcal'),
\]
then $\Kcal^{\rm tw}$ and $(\Kcal')^{\rm tw}$ are kernel-isomorphic.
\end{enumerate}
Consequently, two row-column-mark pure marked colored kernel spaces with the
same infinite marked colored array law are kernel-isomorphic.
\end{theorem}

\begin{proof}
We reduce the marked colored statement to
Theorem~\ref{thm:vershik-pure-function}.  Since $A_0$ and $C$ are countable
and all marks and kernels are bounded, choose compact intervals
$J_\alpha,I_c\subset\R$ such that $m_\alpha(\Omega)\subset J_\alpha$ and
$k_c(\Omega^2)\subset I_c$ modulo null sets.  Set
\[
  A=\left(\prod_{\alpha\in A_0}J_\alpha\right)^2
    \times \prod_{c\in C} I_c .
\]
The product $A$ is compact metrizable and therefore standard Borel.  Define
the product-valued measurable kernel
\[
  F_\Kcal(z,z')=
  \bigl((m_\alpha(z))_{\alpha\in A_0},
        (m_\alpha(z'))_{\alpha\in A_0},
        (k_c(z,z'))_{c\in C}\bigr)\in A .
\]
The off-diagonal $A$-valued matrix
\[
  (F_\Kcal(Z_i,Z_j))_{i\ne j}
\]
contains exactly the same information as the collection of vertex marks and
off-diagonal colored matrices
\[
  (m_\alpha(Z_i))_{\alpha\in A_0,\ i\ge1},
  \qquad
  (k_c(Z_i,Z_j))_{c\in C,\ i\ne j}.
\]
Indeed, each mark $m_\alpha(Z_i)$ appears in every off-diagonal entry
$F_\Kcal(Z_i,Z_j)$ with $j\ne i$, and the coordinate projections of $A$ recover
all marks and all colors.  Conversely, the off-diagonal $A$-valued matrix is
obtained from the marked colored array by applying these coordinate maps.
Thus the marked colored array law and the off-diagonal $A$-valued matrix
distribution determine one another.

The $F_\Kcal$-twin relation is precisely the row-column-mark twin relation.
Indeed,
\[
  F_\Kcal(z,w)=F_\Kcal(z',w)
\]
implies both $m_\alpha(z)=m_\alpha(z')$ for every mark and
$k_c(z,w)=k_c(z',w)$ for every color; the corresponding column equality gives
$k_c(w,z)=k_c(w,z')$ for every color.  Conversely, these mark, row, and column
equalities imply equality of the product-valued rows and columns.  Hence
Theorem~\ref{thm:vershik-pure-function} gives the twin-free quotient
$(\Omega_F,\rho_F,F_\Kcal^{\rm tw})$ of the product-valued function.
Equivalently, if $q:\Omega\to\Omega_F$ is the quotient map, define
\[
  m_\alpha^{\rm tw}(qz)=m_\alpha(z),
  \qquad
  k_c^{\rm tw}(qz,qz')=k_c(z,z')
\]
outside the null sets on which the quotient is not represented.  The preceding
identification of the twin relation shows that these definitions are
well-defined modulo $\rho_F$ and $\rho_F\otimes\rho_F$.  These marks and
kernels are the coordinate projections of $F_\Kcal^{\rm tw}$, and they define
$\Kcal^{\rm tw}$.  Since the marked colored array is a measurable function of
the off-diagonal $A$-valued matrix, the equality
\[
  \mathfrak M(F_\Kcal^{\rm tw})=\mathfrak M(F_\Kcal)
\]
implies $\mathfrak A(\Kcal^{\rm tw})=\mathfrak A(\Kcal)$.

If $\Kcal$ is already row-column-mark pure, the quotient map in
Theorem~\ref{thm:vershik-pure-function} is one-to-one modulo null sets, so
$\Kcal^{\rm tw}$ is kernel-isomorphic to $\Kcal$.  If
$\mathfrak A(\Kcal)=\mathfrak A(\Kcal')$, then the marked colored array laws
agree, so the $A$-valued matrix distributions of
$F_\Kcal$ and $F_{\Kcal'}$ agree.  Theorem~\ref{thm:vershik-pure-function}
gives an isomorphism between the twin-free $A$-valued representatives, and
applying the coordinate projections of $A$ says exactly that every mark and
every colored kernel is preserved.  This proves all assertions.

This also explains the role of diagonal entries in the distance arrays.  The
Vershik theorem reconstructs the off-diagonal kernel up to product-a.e.
equality, while the same-particle diagonal blocks in this paper are not treated
as kernel diagonal values.  They are exactly the one-point marks
$m_\alpha(Z_i)$ and are therefore preserved by the marked reconstruction.
\end{proof}

\section{Unanchored reconstruction of joining projections}

For a joining $\lambda\in\J(T,S)$ define the associated marked colored kernel
space
\[
  \Kcal_\lambda
  =
  \bigl(X\times Y,\lambda,
    (M_X^{a,b})_{a,b\ge0},(M_Y^{a,b})_{a,b\ge0},
    (K_X^{a,b})_{a,b\ge0},(K_Y^{a,b})_{a,b\ge0}\bigr),
\]
where
\[
  M_X^{a,b}(x,y)=d_X(T^a x,T^b x),\qquad
  M_Y^{a,b}(x,y)=d_Y(S^a y,S^b y),
\]
and
\[
  K_X^{a,b}((x,y),(x',y'))=d_X(T^a x,T^b x'),\qquad
  K_Y^{a,b}((x,y),(x',y'))=d_Y(S^a y,S^b y').
\]
The marks are the same-particle orbit-distance blocks in $\Phi_{n,m}$; the
kernels are evaluated on distinct sampled particles.

\begin{definition}
The joining $\lambda$ is \emph{distance-array pure} if the marked colored
kernel space $\Kcal_\lambda$ is row-column-mark pure.  Its twin-free quotient is denoted by
$\Kcal_\lambda^{\rm tw}$.
\end{definition}

\begin{theorem}[Intrinsic reconstruction of joining kernels]
\label{thm:joining-kernel-reconstruction}
Let $\lambda,\lambda'\in\J(T,S)$.  If
\[
  \Phi_{n,m}(\lambda)=\Phi_{n,m}(\lambda')
  \qquad\text{for every } n,m\ge1,
\]
then $\Kcal_\lambda^{\rm tw}$ and $\Kcal_{\lambda'}^{\rm tw}$ are
kernel-isomorphic.  If, in addition, $\lambda$ and $\lambda'$ are
distance-array pure, then there is a measure-space isomorphism
\[
  U:(X\times Y,\lambda)\to(X\times Y,\lambda')
\]
such that for all $a,b\ge0$,
\[
  M_X^{a,b}(z)=M_X^{a,b}(Uz),\qquad
  M_Y^{a,b}(z)=M_Y^{a,b}(Uz)
\]
for $\lambda$-almost every $z$, and
\[
  K_X^{a,b}(z,z')=K_X^{a,b}(Uz,Uz'),\qquad
  K_Y^{a,b}(z,z')=K_Y^{a,b}(Uz,Uz')
\]
for $\lambda\otimes\lambda$-almost every $(z,z')$.
\end{theorem}

\begin{proof}
The family of finite laws $\Phi_{n,m}(\lambda)$ determines exactly the
finite-dimensional distributions of the infinite marked colored array generated
by $\Kcal_\lambda$.  Indeed, any cylinder event in the infinite array involves
only finitely many sampled particles and finitely many colors and marks, hence
is contained in the array recorded by $\Phi_{n,m}$ for $n$ and $m$ large
enough.  In these finite arrays the diagonal blocks $i=j$ encode the vertex
marks $M_X^{a,b}$ and $M_Y^{a,b}$, while the off-diagonal blocks $i\ne j$
encode the colored kernels $K_X^{a,b}$ and $K_Y^{a,b}$.  Equality for all
$n,m$ therefore implies
$\mathfrak A(\Kcal_\lambda)=\mathfrak A(\Kcal_{\lambda'})$.  The twin-free quotient
claim follows from Theorem~\ref{thm:kernel-reconstruction}.  If both spaces are
already row-column-mark pure, the last assertion is the final assertion of that
theorem.
\end{proof}

\begin{remark}
The equality above uses the full distance arrays, including the entries with
$i=j$.  In the marked-kernel formulation these same-particle entries are the
vertex marks $M_X^{a,b}$ and $M_Y^{a,b}$, while the entries with $i\ne j$ are
the edge kernels $K_X^{a,b}$ and $K_Y^{a,b}$.  Thus a kernel-isomorphism of
marked colored spaces preserves the full finite laws $\Phi_{n,m}$, not merely
their off-diagonal parts.
\end{remark}

The preceding theorem reconstructs the joining as an intrinsic marked colored kernel
space.  To identify the ambiguity with familiar coordinate automorphisms one
needs a genuine coordinate-generation condition.

\begin{definition}
Let $\lambda\in\J(T,S)$.  The \emph{$X$-orbit-distance factor} of
$(X\times Y,\lambda)$ is the completed sub-$\sigma$-algebra
$\mathcal F_X^\lambda\subset\mathcal B(X\times Y)$ generated by the
$X$-colored kernels in either variable: it is the smallest complete
sub-$\sigma$-algebra $\mathcal G$ such that, for every $a,b\ge0$, the function
$K_X^{a,b}$ is both
$\mathcal G\otimes\mathcal B(X\times Y)$-measurable and
$\mathcal B(X\times Y)\otimes\mathcal G$-measurable modulo
$\lambda\otimes\lambda$.  Define $\mathcal F_Y^\lambda$ analogously.  We say
that $\lambda$ has \emph{coordinate-generated orbit-distance factors} if
\[
  \mathcal F_X^\lambda=\pi_X^{-1}\mathcal B_X,
  \qquad
  \mathcal F_Y^\lambda=\pi_Y^{-1}\mathcal B_Y
\]
modulo $\lambda$.
\end{definition}

\begin{lemma}[Full support generates the coordinate factors]
\label{lem:full-support-generation}
In a support-reduced model, every joining $\lambda\in\J(T,S)$ has
coordinate-generated orbit-distance factors.
\end{lemma}

\begin{proof}
By the support convention, $\mu$ and $\nu$ have full support.  We prove the
assertion for $X$; the proof for $Y$ is identical.  All generated
$\sigma$-algebras are understood modulo $\lambda$-null sets, or equivalently
inside the completed measure algebra of $(X\times Y,\lambda)$.  Every
$X$-colored kernel depends on a point $z=(x,y)$ only through its $X$-coordinate
in each variable.  Therefore
\[
  \mathcal F_X^\lambda\subset \pi_X^{-1}\mathcal B_X
\]
modulo $\lambda$.

For the reverse inclusion, use only the color $(0,0)$:
\[
  K_X^{0,0}(z,z')=d_X(x,x').
\]
Since this kernel is
$\mathcal F_X^\lambda\otimes\mathcal B(X\times Y)$-measurable, Fubini gives a
set $E\subset X\times Y$ with $\lambda(E)=1$ such that for every
$z'=(x',y')\in E$ the section
\[
  z\longmapsto d_X(\pi_X z,x')
\]
is $\mathcal F_X^\lambda$-measurable modulo $\lambda$.  The projection
$D=\pi_X(E)$ has $\mu(D)=1$.  Because $\mu$ has full support, $D$ is dense in
$X$.  Choose a countable dense subset $(x_r)_{r\ge1}$ of $D$, and for each
$r$ choose $y_r$ with $(x_r,y_r)\in E$.  Then all functions
$x\mapsto d_X(x,x_r)$, pulled back by $\pi_X$, are
$\mathcal F_X^\lambda$-measurable modulo $\lambda$, because they are sections of
$K_X^{0,0}$ at the points $(x_r,y_r)\in E$.

The map
\[
  x\longmapsto (d_X(x,x_r))_{r\ge1}
\]
is injective on the compact metric space $X$ because $(x_r)$ is dense.  By
Lemma~\ref{lem:distance-coordinates-generate}, these distance functions
generate $\mathcal B_X$.  Pulling back by $\pi_X$ gives
$\pi_X^{-1}\mathcal B_X\subset\mathcal F_X^\lambda$ modulo $\lambda$.  This
proves equality.
\end{proof}

\begin{lemma}[Full-support distance separation]
\label{lem:full-support-separation}
Let $(X,d)$ be compact and let $\mu$ be a Borel probability measure with full
support.  If $u,v:X\to X$ are measurable maps such that
\[
  d(u(x),x')=d(v(x),x')
\]
for $\mu\otimes\mu$-almost every $(x,x')$, then $u(x)=v(x)$ for
$\mu$-almost every $x$.
\end{lemma}

\begin{proof}
By Fubini, for $\mu$-almost every $x$ there is a Borel set $D_x\subset X$ with
$\mu(D_x)=1$ such that
\[
  d(u(x),x')=d(v(x),x')\qquad\forall x'\in D_x .
\]
Every full-measure Borel subset of $X$ is dense: if a nonempty open set were
disjoint from it, that open set would have measure zero, contradicting full
support.  Hence $D_x$ is dense.  The functions
$x'\mapsto d(u(x),x')$ and $x'\mapsto d(v(x),x')$ are continuous, so equality
on $D_x$ extends to all $x'\in X$.  Taking $x'=u(x)$ gives
\[
  0=d(u(x),u(x))=d(v(x),u(x)),
\]
and therefore $u(x)=v(x)$.  This holds for $\mu$-almost every $x$.
\end{proof}

\begin{lemma}[Almost-everywhere isometries are automorphisms]
\label{lem:ae-isometry-automorphism}
Let $(X,d)$ be compact and let $\mu$ have full support.  Suppose
$A:X\to X$ is measurable, $A_\#\mu=\mu$, and
\[
  d(Ax,Ax')=d(x,x')
\]
for $\mu\otimes\mu$-almost every $(x,x')$.  Then $A$ is a measure-space
automorphism of $(X,\mu)$.
\end{lemma}

\begin{proof}
Let $\sigma(A)=A^{-1}\mathcal B_X$ be the sub-$\sigma$-algebra generated by
$A$.  By Fubini, there is a full-measure set $D\subset X$ such that for every
$y\in D$,
\[
  d(Ax,Ay)=d(x,y)
\]
for $\mu$-almost every $x$.  For each such $y$, the function
$x\mapsto d(x,y)$ is $\sigma(A)$-measurable, because it agrees almost
everywhere with $x\mapsto d(Ax,Ay)$.  The set $D$ is dense by full support, so
we may choose a countable dense subset $(y_r)_{r\ge1}$ of $D$.  The map
$x\mapsto(d(x,y_r))_{r\ge1}$ is injective and Borel on the compact space $X$;
hence the functions $d(\cdot,y_r)$ generate $\mathcal B_X$.  Therefore
\[
  \mathcal B_X\subset \sigma(A)
\]
modulo $\mu$.  The reverse inclusion is automatic, so
$\sigma(A)=\mathcal B_X$ modulo $\mu$.

The pullback map $B\mapsto A^{-1}B$ is measure preserving because
$A_\#\mu=\mu$, and the preceding equality says that it is onto the completed
measure algebra.  It is also injective: if $\mu(A^{-1}B)=0$, then
$\mu(B)=0$.  Thus $A$ induces an automorphism of the measure algebra, which is
equivalent on a standard probability space to being a measure-space
automorphism modulo null sets.  Since compact metric probability spaces are
standard Borel, this measure-algebra automorphism induced by $A$ admits a
measurable inverse modulo null sets.
\end{proof}

\begin{definition}
Let $\Aut_{\rm md}(\Xbf)$ be the group of measure-space automorphisms $A$ of
$(X,\mu)$ such that $A\circ T=T\circ A$ modulo $\mu$ and
\[
  d_X(Ax,Ax')=d_X(x,x')
\]
for $\mu\otimes\mu$-almost every $(x,x')$.  Define $\Aut_{\rm md}(\Ybf)$
analogously.  Two joinings $\lambda,\lambda'$ are
\emph{coordinate-isomorphic} if
\[
  \lambda'=(A\times B)_\#\lambda
\]
for some $A\in\Aut_{\rm md}(\Xbf)$ and $B\in\Aut_{\rm md}(\Ybf)$.
\end{definition}

\begin{theorem}[Splitting of kernel isomorphisms]
\label{thm:splitting}
Assume $T,S$ are continuous and the models are support-reduced, and
$\lambda,\lambda'\in\J(T,S)$.  Let
\[
  U:(X\times Y,\lambda)\to(X\times Y,\lambda')
\]
be a kernel-preserving isomorphism between $\Kcal_\lambda$ and
$\Kcal_{\lambda'}$.  Then there are
$A\in\Aut_{\rm md}(\Xbf)$ and $B\in\Aut_{\rm md}(\Ybf)$ such that
\[
  U(x,y)=(Ax,By)
\]
for $\lambda$-almost every $(x,y)$.  Consequently
$\lambda'=(A\times B)_\#\lambda$.
\end{theorem}

\begin{proof}
By Lemma~\ref{lem:full-support-generation}, both $\lambda$ and $\lambda'$ have
coordinate-generated orbit-distance factors.
Let
\[
  H=U^{-1}\mathcal F_X^{\lambda'}
\]
be the pullback of the $X$-orbit-distance factor of
$(X\times Y,\lambda')$.  Since $U$ is kernel-preserving,
\[
  K_X^{a,b}(z,z')=K_X^{a,b}(Uz,Uz')
\]
modulo $\lambda\otimes\lambda$ for every $a,b\ge0$.  Hence each
$X$-colored kernel of $\Kcal_\lambda$ is measurable with respect to $H$ in the
corresponding variable.  By the minimality in the definition of
$\mathcal F_X^\lambda$,
\[
  \mathcal F_X^\lambda\subset U^{-1}\mathcal F_X^{\lambda'} .
\]
Applying the same argument to the inverse kernel isomorphism $U^{-1}$ gives
the reverse inclusion; $U^{-1}$ is kernel-preserving because the defining
kernel equalities for $U$ pull back under the measure-space isomorphism $U$.
Hence
\[
  U^{-1}\mathcal F_X^{\lambda'}=\mathcal F_X^\lambda
\]
in the completed measure algebra.  The coordinate-generation hypothesis gives
\[
  U^{-1}\pi_X^{-1}\mathcal B_X=\pi_X^{-1}\mathcal B_X
\]
modulo $\lambda$.  Hence $\pi_X\circ U$ is
$\pi_X^{-1}\mathcal B_X$-measurable.  By the Doob--Dynkin factorization lemma
on standard Borel spaces, there is a measurable map $A:X\to X$ such that
\[
  \pi_X(U(x,y))=A x
\]
for $\lambda$-almost every $(x,y)$.  The same argument applied to the
$Y$-colored kernels gives a measurable $B:Y\to Y$ with
$\pi_Y(U(x,y))=B y$ almost everywhere.  Thus $U=A\times B$ modulo $\lambda$.

Since $U_\#\lambda=\lambda'$ and both joinings have marginals $\mu,\nu$, the
maps $A$ and $B$ preserve $\mu$ and $\nu$.  The preservation of the color
$(a,b)=(0,0)$ gives
\[
  d_X(Ax,Ax')=d_X(x,x')
\]
for $\mu\otimes\mu$-almost every $(x,x')$, because the $X$-marginal of
$\lambda\otimes\lambda$ is $\mu\otimes\mu$.  Similarly $B$ is an isometry
modulo $\nu$.  By Lemma~\ref{lem:ae-isometry-automorphism}, $A$ and $B$ are
measure-space automorphisms of $(X,\mu)$ and $(Y,\nu)$, respectively.

It remains to prove that $A$ commutes with $T$ modulo $\mu$.  Kernel
preservation for the color $(1,0)$ says that, for
$z=(x,y)$ and $z'=(x',y')$,
\[
  K_X^{1,0}(Uz,Uz')=K_X^{1,0}(z,z')
\]
for $\lambda\otimes\lambda$-almost every $(z,z')$.  Since $U=A\times B$
modulo $\lambda$, this is exactly
\[
  d_X(TAx,Ax')=d_X(Tx,x')
\]
for $\mu\otimes\mu$-almost every $(x,x')$.  Since $A$ is an isometry modulo
$\mu$ and $T_\#\mu=\mu$, applying
$d_X(Au,Av)=d_X(u,v)$ with $(u,v)=(Tx,x')$ gives
\[
  d_X(ATx,Ax')=d_X(Tx,x')
\]
for $\mu\otimes\mu$-almost every $(x,x')$.  Hence
$d_X(TAx,Ax')=d_X(ATx,Ax')$ for $\mu\otimes\mu$-almost every $(x,x')$.
Since $A_\#\mu=\mu$, the map $(x,x')\mapsto(x,Ax')$ sends
$\mu\otimes\mu$ to $\mu\otimes\mu$.  Therefore
\[
  d_X(TAx,y)=d_X(ATx,y)
\]
for $\mu\otimes\mu$-almost every $(x,y)$.  Applying
Lemma~\ref{lem:full-support-separation} to the maps $TA$ and $AT$
gives $TAx=ATx$ for $\mu$-almost every $x$.  The proof for $B$ and $S$ is
identical.  Thus $A\in\Aut_{\rm md}(\Xbf)$ and $B\in\Aut_{\rm md}(\Ybf)$.
\end{proof}

\begin{corollary}[Reconstruction modulo coordinate symmetries]
\label{cor:mod-auto}
Assume $T,S$ are continuous and the models are support-reduced, and
$\lambda,\lambda'\in\J(T,S)$ are distance-array pure.  If
$\Phi_{n,m}(\lambda)=\Phi_{n,m}(\lambda')$ for all $n,m$, then $\lambda$ and
$\lambda'$ are coordinate-isomorphic.
\end{corollary}

\begin{proof}
Theorem~\ref{thm:joining-kernel-reconstruction} gives a kernel-preserving
isomorphism $U$.  Theorem~\ref{thm:splitting} identifies $U$ with a product
$A\times B$ of metric-dynamical automorphisms.  Hence
$\lambda'=(A\times B)_\#\lambda$.
\end{proof}

\begin{remark}
The coordinate-generation condition is the point at which an intrinsic
marked-kernel isomorphism is converted into a statement about the original coordinates.
Without support reduction, the same statement should be read on
$\supp\mu\times\supp\nu$.  The twin-free quotient theorem is the correct
intrinsic conclusion: unanchored arrays reconstruct the observable marked
orbit-distance kernel structure, not arbitrary labels outside the support of a
chosen model.
\end{remark}

\section{Anchored fixed-model reconstruction}

We now add anchors to obtain literal equality of joinings on a fixed compact
model.

\begin{definition}
Let $(a_r)_{r\ge1}$ be a dense sequence in $\supp\mu$ and $(b_r)_{r\ge1}$ a
dense sequence in $\supp\nu$.  For $R\ge1$ define the anchored array
\[
 \widetilde\Dcal^X_{n,m,R}(z_1,\ldots,z_n)
 =
 \left(
   \Dcal^X_{n,m}(z_1,\ldots,z_n),
   \bigl(d_X(T^a x_i,a_r)\bigr)_{
      1\le i\le n,\ 0\le a<m,\ 1\le r\le R}
 \right),
\]
and analogously $\widetilde\Dcal^Y_{n,m,R}$.  The anchored projection is
\[
 \widetilde\Phi_{n,m,R}(\lambda)
 =
 \Law_{\lambda^{\otimes n}}\bigl(
   \widetilde\Dcal^X_{n,m,R},
   \widetilde\Dcal^Y_{n,m,R}
 \bigr).
\]
\end{definition}

\begin{lemma}[Metric anchors separate support points]
\label{lem:anchors}
Let $(X,d)$ be compact, let $K\subset X$ be compact, and let $(a_r)$ be dense
in $K$.  The map
\[
  \iota_K:K\to\R^\N,\qquad
  \iota_K(x)=(d(x,a_r))_{r\ge1},
\]
is a continuous injective map.  Hence it is a Borel isomorphism from $K$ onto
its image and the functions $x\mapsto d(x,a_r)$, restricted to $K$, generate
$\mathcal B(K)$.
\end{lemma}

\begin{proof}
Continuity is immediate.  If $x\neq x'$, choose $r_k$ with $a_{r_k}\to x$.
Then $d(x,a_{r_k})\to0$ while
$d(x',a_{r_k})\to d(x',x)>0$, so the coordinate sequences differ.  A continuous
injective map from a compact space into a Hausdorff space is a homeomorphism
onto its image.  Therefore the coordinate functions generate the Borel
$\sigma$-algebra.
\end{proof}

\begin{theorem}[Anchored fixed-model reconstruction]
\label{thm:anchored}
Let $\lambda,\lambda'\in\J(T,S)$.  If
\[
  \widetilde\Phi_{n,m,R}(\lambda)
  =
  \widetilde\Phi_{n,m,R}(\lambda')
  \qquad\text{for all } n,m,R\ge1,
\]
then $\lambda=\lambda'$.
\end{theorem}

\begin{proof}
It is enough to use the case $n=m=1$ and all $R$.  For $z=(x,y)$,
$\widetilde\Dcal^X_{1,1,R}$ records
$(d_X(x,a_r))_{1\le r\le R}$ and $\widetilde\Dcal^Y_{1,1,R}$ records
$(d_Y(y,b_r))_{1\le r\le R}$.  Write
\[
  \iota_X(x)=(d_X(x,a_r))_{r\ge1}\quad (x\in\supp\mu),
  \qquad
  \iota_Y(y)=(d_Y(y,b_r))_{r\ge1}\quad (y\in\supp\nu).
\]
Equality of the laws for all $R$ gives equality of all finite-dimensional
distributions of $(\iota_X,\iota_Y)$, and therefore
\[
  (\iota_X,\iota_Y)_\#\lambda=(\iota_X,\iota_Y)_\#\lambda'
\]
on
\[
  \iota_X(\supp\mu)\times\iota_Y(\supp\nu).
\]
Here probability measures on the countable product of real lines are
determined by their finite-dimensional cylinder distributions.
Both $\lambda$ and $\lambda'$ are supported on $\supp\mu\times\supp\nu$.  By
Lemma~\ref{lem:anchors}, applied to the two supports, $\iota_X$ and $\iota_Y$
are Borel isomorphisms onto their support images.  Pulling back the common
pushforward measure gives $\lambda=\lambda'$.
\end{proof}

\begin{remark}
The proof shows why anchored reconstruction is a fixed-coordinate safety net
rather than the main intrinsic theorem: anchors deliberately encode the points
of the model.  The unanchored theorem reconstructs only the marked colored
kernel structure and therefore has the correct quotient ambiguity.
\end{remark}

\section{Distance-array criteria for disjointness}

\begin{theorem}[Anchored distance-array criterion for disjointness]
\label{thm:anchored-disjoint}
Let $\Xbf$ and $\Ybf$ be compact metric measure-preserving systems with dense
anchor families.  Then the following are equivalent:
\begin{enumerate}
\item $\Xbf$ and $\Ybf$ are disjoint;
\item for every $n,m,R\ge1$,
\[
  \widetilde\Phi_{n,m,R}(\J(T,S))
  =
  \{\widetilde\Phi_{n,m,R}(\mu\otimes\nu)\}.
\]
\end{enumerate}
\end{theorem}

\begin{proof}
If the systems are disjoint, the joining set consists only of
$\mu\otimes\nu$, so the displayed equality is immediate.  Conversely, assume
the displayed equality for all $n,m,R$.  Let $\lambda\in\J(T,S)$.  Then
\[
  \widetilde\Phi_{n,m,R}(\lambda)
  =
  \widetilde\Phi_{n,m,R}(\mu\otimes\nu)
  \qquad\forall n,m,R.
\]
By Theorem~\ref{thm:anchored}, $\lambda=\mu\otimes\nu$.  Since $\lambda$ was
arbitrary, $\J(T,S)=\{\mu\otimes\nu\}$.
\end{proof}

For the unanchored formulation, let $\J(T,S)/{\sim}$ denote the quotient
obtained by identifying joinings whose twin-free marked colored-kernel quotients are
kernel-isomorphic.  The product class is denoted by $[\mu\otimes\nu]$.
The next statement is the intrinsic quotient-level analogue of the anchored
disjointness criterion above.  It is not meant to be a fixed-model formulation
of Furstenberg disjointness: unanchored arrays do not remember labels beyond
the marked kernel quotient.

\begin{theorem}[Intrinsic disjointness criterion modulo symmetries]
\label{thm:intrinsic-disjoint}
The following are equivalent:
\begin{enumerate}
\item $\J(T,S)/{\sim}=\{[\mu\otimes\nu]\}$;
\item for every $n,m\ge1$,
\[
  \Phi_{n,m}(\J(T,S))
  =
  \{\Phi_{n,m}(\mu\otimes\nu)\}.
\]
\end{enumerate}
\end{theorem}

\begin{proof}
Assume first that $\J(T,S)/{\sim}=\{[\mu\otimes\nu]\}$, and let
$\lambda\in\J(T,S)$.  Then $\lambda\sim\mu\otimes\nu$, so their twin-free
marked colored-kernel quotients are kernel-isomorphic.  Kernel-isomorphic
marked colored spaces have the same infinite marked colored array law.  Since
each space and its twin-free quotient have the same array law by
Theorem~\ref{thm:kernel-reconstruction}, the original joining kernels
$\Kcal_\lambda$ and $\Kcal_{\mu\otimes\nu}$ also have the same infinite marked
colored array law, and hence the same finite-dimensional laws.  Therefore
\[
  \Phi_{n,m}(\lambda)=\Phi_{n,m}(\mu\otimes\nu)
  \qquad\forall n,m\ge1.
\]
Since $\lambda$ was arbitrary, the second condition follows.  Conversely, suppose
the second condition holds.  For any $\lambda\in\J(T,S)$, all finite array
projections agree with those of $\mu\otimes\nu$.  By
Theorem~\ref{thm:joining-kernel-reconstruction}, their twin-free quotients are
kernel-isomorphic, so $[\lambda]=[\mu\otimes\nu]$.

\end{proof}

\begin{remark}
The role of Theorem~\ref{thm:intrinsic-disjoint} is to say exactly what the
unanchored observables can determine: the collapse of the joining space after
passing to twin-free marked-kernel quotients.  Ordinary Furstenberg
disjointness is stronger because it asks for literal equality of joinings on
the fixed product model.  That fixed-model statement is Theorem~\ref{thm:anchored-disjoint},
where anchors remove the quotient ambiguity.
\end{remark}

\begin{remark}
Finite profiles depend on the chosen metric model.  The all-order anchored
zero criterion does not: by Theorem~\ref{thm:anchored-disjoint}, it is exactly
Furstenberg disjointness, which is a measure-theoretic property of the systems.
\end{remark}

\section{Quantitative finite-pattern dependence}

Let $W_p$ denote the ordinary $p$-Wasserstein distance on the finite array cube,
with any fixed product Euclidean metric.  Compactness ensures finiteness.

\begin{definition}
For $p\ge1$ define the unanchored maximal finite-pattern dependence coefficient
\[
  \operatorname{Dep}_{n,m,p}(\Xbf,\Ybf)
  =
  \sup_{\lambda\in\J(T,S)}
  W_p\bigl(
    \Phi_{n,m}(\lambda),
    \Phi_{n,m}(\mu\otimes\nu)
  \bigr).
\]
The anchored coefficient is
\[
  \widetilde{\operatorname{Dep}}_{n,m,R,p}(\Xbf,\Ybf)
  =
  \sup_{\lambda\in\J(T,S)}
  W_p\bigl(
    \widetilde\Phi_{n,m,R}(\lambda),
    \widetilde\Phi_{n,m,R}(\mu\otimes\nu)
  \bigr).
\]
\end{definition}

\begin{proposition}[Existence of maximizers]
\label{prop:max}
If $T$ and $S$ are continuous, then the suprema defining
$\operatorname{Dep}_{n,m,p}$ and
$\widetilde{\operatorname{Dep}}_{n,m,R,p}$ are attained.
\end{proposition}

\begin{proof}
By Proposition~\ref{prop:compact}, $\lambda\mapsto\Phi_{n,m}(\lambda)$ is
continuous on the compact joining set.  Wasserstein distance on a compact
metric space metrizes weak convergence and is continuous
\cite[Ch.~6]{villani2009optimal}.  Hence the displayed
objective is continuous and attains its maximum.  The anchored case is
identical because the anchored array map is continuous when the dynamics are
continuous.
\end{proof}

\begin{theorem}[Quantitative zero criterion]
\label{thm:zero}
For compact metric measure-preserving systems with dense support anchors and
any fixed $p\ge1$,
\[
  \widetilde{\operatorname{Dep}}_{n,m,R,p}(\Xbf,\Ybf)=0
  \qquad\forall n,m,R\ge1
\]
if and only if $\Xbf$ and $\Ybf$ are disjoint.
\end{theorem}

\begin{proof}
If $\Xbf$ and $\Ybf$ are disjoint, then $\J(T,S)=\{\mu\otimes\nu\}$.  Hence the
supremum defining each anchored coefficient is taken over the single product
joining, and every coefficient is zero.

Conversely, assume that
\[
  \widetilde{\operatorname{Dep}}_{n,m,R,p}(\Xbf,\Ybf)=0
  \qquad\forall n,m,R\ge1 .
\]
Let $\lambda\in\J(T,S)$.  Since the coefficient is the supremum over all
joinings, for every $n,m,R$ we have
\[
  W_p\bigl(
    \widetilde\Phi_{n,m,R}(\lambda),
    \widetilde\Phi_{n,m,R}(\mu\otimes\nu)
  \bigr)=0 .
\]
The finite anchored array cube is compact, and $W_p$ is a metric on
probability measures on this cube.  Therefore
\[
  \widetilde\Phi_{n,m,R}(\lambda)
  =
  \widetilde\Phi_{n,m,R}(\mu\otimes\nu)
  \qquad\forall n,m,R\ge1 .
\]
By the anchored disjointness criterion, Theorem~\ref{thm:anchored-disjoint},
this implies $\lambda=\mu\otimes\nu$.  Since $\lambda$ was arbitrary,
$\J(T,S)=\{\mu\otimes\nu\}$.
\end{proof}

The raw coefficients need not decay with the orbit length: finite arrays always
contain the short-time coordinates already visible at $m=1$.  Decay estimates
therefore become meaningful after applying normalized orbit-observable
projections to the array.  The next examples separate these two phenomena.  The
Bernoulli example shows uniform non-decay of the raw profile, while the
expanding and Markov examples give explicit Wasserstein rates for normalized
time-averaged projections of the same distance-array laws.

\begin{proposition}[Uniform non-decay for Bernoulli shifts]
\label{prop:bernoulli-unanchored-bound}
Let $X=\{0,1\}^{\mathbb Z}$, let $\mu=(\frac12,\frac12)^{\mathbb Z}$, let
$T$ be the two-sided Bernoulli shift, and equip $X$ with the metric
\[
  d(x,x')=\sum_{k\in\mathbb Z}2^{-|k|-2}|x_k-x'_k|.
\]
Let $\Xbf=(X,d,\mu,T)$.  If the finite array cubes are equipped with the
coordinate Euclidean metrics used in the definition of $W_p$, then for every
$m\ge1$ and every $p\ge1$,
\[
  \operatorname{Dep}_{2,m,p}(\Xbf,\Xbf)
  \ge
  \frac{5}{144\sqrt2}.
\]
In particular the unanchored dependence profile of the self-pair
$(\Xbf,\Xbf)$ has no decay in the orbit length $m$ at particle level $n=2$.
\end{proposition}

\begin{proof}
Let $\lambda_\Delta=(x\mapsto(x,x))_\#\mu$ be the diagonal joining of
$\Xbf$ with itself.  Since $\lambda_\Delta\in\J(T,T)$,
\[
  \operatorname{Dep}_{2,m,p}(\Xbf,\Xbf)
  \ge
  W_p\bigl(\Phi_{2,m}(\lambda_\Delta),\Phi_{2,m}(\mu\otimes\mu)\bigr).
\]
It is enough to give a lower bound for $W_1$, because $W_p\ge W_1$ on a
bounded metric space for every $p\ge1$.

The array $\Phi_{2,m}$ contains, for every $m\ge1$, the two coordinates
\[
  r=d(x_1,x_2),
  \qquad
  s=d(y_1,y_2),
\]
corresponding to $a=b=0$.  Let
\[
  F(r,s)=(r-s)^2 .
\]
The diameter of $(X,d)$ is at most
\[
  \sum_{k\in\mathbb Z}2^{-|k|-2}
  =
  \frac14+2\sum_{k\ge1}2^{-k-2}
  =
  \frac34 .
\]
Therefore $0\le r,s\le 3/4$, and the gradient of $F$ satisfies
\[
  \|\nabla F(r,s)\|_2
  =
  2\sqrt2\,|r-s|
  \le
  \frac{3\sqrt2}{2}.
\]
Thus $F$ is $(3\sqrt2/2)$-Lipschitz as a function of the full finite array
with its coordinate Euclidean metric.  Hence
$\psi=(2/(3\sqrt2))F$ is $1$-Lipschitz.

Under the diagonal joining, $y_i=x_i$ for $i=1,2$, so $r=s$ almost surely and
\[
  \int \psi\,\dd\Phi_{2,m}(\lambda_\Delta)=0.
\]
Under the product joining, $r$ and $s$ are independent copies of
\[
  R=d(X_1,X_2)
  =
  \sum_{k\in\mathbb Z}2^{-|k|-2}Z_k,
\]
where the $Z_k=|X_{1,k}-X_{2,k}|$ are independent Bernoulli variables with
parameter $1/2$.  Therefore
\[
  \operatorname{Var}(R)
  =
  \sum_{k\in\mathbb Z}2^{-2|k|-4}\operatorname{Var}(Z_k)
  =
  \frac14\sum_{k\in\mathbb Z}2^{-2|k|-4}.
\]
The remaining geometric sum is
\[
  \sum_{k\in\mathbb Z}2^{-2|k|-4}
  =
  \frac1{16}\left(1+2\sum_{k\ge1}4^{-k}\right)
  =
  \frac1{16}\left(1+\frac23\right)
  =
  \frac5{48}.
\]
Hence $\operatorname{Var}(R)=5/192$.  If $R'$ is an independent copy of $R$,
then
\[
  \mathbb E(R-R')^2=2\operatorname{Var}(R)=\frac5{96}.
\]
Consequently,
\[
  \int \psi\,\dd\Phi_{2,m}(\mu\otimes\mu)
  =
  \frac{2}{3\sqrt2}\cdot\frac5{96}
  =
  \frac5{144\sqrt2}.
\]
By the Kantorovich--Rubinstein dual formula for $W_1$
\cite[Ch.~1]{villani2009optimal},
\[
  W_1\bigl(\Phi_{2,m}(\lambda_\Delta),\Phi_{2,m}(\mu\otimes\mu)\bigr)
  \ge
  \frac5{144\sqrt2}.
\]
This proves the displayed lower bound for all $m\ge1$ and all $p\ge1$.
\end{proof}

\begin{proposition}[An averaged Wasserstein rate for the doubling map]
\label{prop:doubling-averaged-rate}
Let $\mathbb T=\mathbb R/\mathbb Z$, let $T x=2x\pmod 1$, let $\mu$ be Haar
measure, and let $d(x,y)=\min_{k\in\mathbb Z}|x-y-k|$ be the circle metric.
For $m\ge1$ define the orbit-averaged two-particle statistic
\[
  A_m(x_1,x_2)=\frac1m\sum_{t=0}^{m-1} d(T^t x_1,T^t x_2).
\]
Let $\lambda_\Delta=(x\mapsto(x,x))_\#\mu$ be the diagonal joining of
$(\mathbb T,d,\mu,T)$ with itself, and let $Q_m^\Delta$ and $Q_m^\otimes$ be
the laws on $[0,1/2]^2$ of
\[
  (A_m(x_1,x_2),A_m(y_1,y_2))
\]
under $\lambda_\Delta^{\otimes2}$ and $(\mu\otimes\mu)^{\otimes2}$,
respectively.  Then
\[
  \frac{1}{24\sqrt2\,m}
  \le
  W_1(Q_m^\Delta,Q_m^\otimes)
  \le
  \frac{1}{\sqrt{24m}} .
\]
Consequently, for the full unanchored array laws with the coordinate Euclidean
array metric,
\[
  W_1\bigl(\Phi_{2,m}(\lambda_\Delta),\Phi_{2,m}(\mu\otimes\mu)\bigr)
  \ge
  \frac{1}{24\sqrt{2m}} .
\]
\end{proposition}

\begin{proof}
Write
\[
  f(u)=d(u,0),\qquad u\in\mathbb T .
\]
If $X_1,X_2$ are independent Haar points, then
$U=X_1-X_2$ is Haar and
\[
  A_m(X_1,X_2)=\frac1m\sum_{t=0}^{m-1} f(2^tU).
\]
The function $f(u)=\min(u,1-u)$ has Fourier expansion
\[
  f(u)=\frac14-\frac{2}{\pi^2}
  \sum_{r=0}^{\infty}
  \frac{\cos(2\pi(2r+1)u)}{(2r+1)^2}.
\]
Thus only odd Fourier frequencies occur.  For every $t\ge1$, the nonzero
frequencies of $f(2^t u)$ are $2^t(2r+1)$, which are even, and hence are
disjoint from the odd frequencies of $f(u)$.  Orthogonality of exponentials
therefore gives
\[
  \operatorname{Cov}\bigl(f(U),f(2^tU)\bigr)=0
  \qquad\forall t\ge1.
\]
More generally, if $0\le s<t$, then the change of variables $v=2^s u$ and
Haar invariance give
\[
  \int_{\mathbb T} f(2^s u)f(2^t u)\,\dd u
  =
  \int_{\mathbb T} f(v)f(2^{t-s}v)\,\dd v,
\]
so the same Fourier orthogonality yields
\[
  \operatorname{Cov}\bigl(f(2^sU),f(2^tU)\bigr)=0 .
\]
Also
\[
  \int_{\mathbb T} f\,\dd u=\frac14,\qquad
  \int_{\mathbb T} f^2\,\dd u
  =
  2\int_0^{1/2}u^2\,\dd u
  =
  \frac1{12},
\]
so
\[
  \operatorname{Var}(f(U))=\frac1{12}-\frac1{16}=\frac1{48}.
\]
The covariance cancellation yields the exact variance identity
\[
  \operatorname{Var}(A_m)=\frac{1}{48m}.
\]

Under $\lambda_\Delta^{\otimes2}$ the two coordinates in $Q_m^\Delta$ are
identical, so $Q_m^\Delta=\Law(A_m,A_m)$.  Under the product joining they are
independent copies, so $Q_m^\otimes=\Law(A_m,A_m')$ with $A_m'$ independent of
$A_m$.

For the upper bound, couple $(A_m,A_m)$ and $(A_m,A_m')$ using the same first
coordinate and an independent copy in the second coordinate.  Then
\[
  W_1(Q_m^\Delta,Q_m^\otimes)
  \le
  \mathbb E|A_m-A_m'|
  \le
  \sqrt{\mathbb E(A_m-A_m')^2}
  =
  \sqrt{2\operatorname{Var}(A_m)}
  =
  \frac{1}{\sqrt{24m}}.
\]
For the lower bound, the function
\[
  \psi(a,b)=\frac{(a-b)^2}{\sqrt2},\qquad 0\le a,b\le\frac12,
\]
is $1$-Lipschitz because
\[
  \|\nabla\psi(a,b)\|_2=\frac{2|a-b|}{\sqrt2}\sqrt2=2|a-b|\le1.
\]
Therefore the Kantorovich--Rubinstein dual formula for $W_1$
\cite[Ch.~1]{villani2009optimal} gives
\[
  W_1(Q_m^\Delta,Q_m^\otimes)
  \ge
  \left|
  \mathbb E\psi(A_m,A_m)-\mathbb E\psi(A_m,A_m')
  \right|
  =
  \frac{1}{\sqrt2}\mathbb E(A_m-A_m')^2
  =
  \frac{1}{24\sqrt2\,m}.
\]

Finally, the observable defining $Q_m^\Delta$ and $Q_m^\otimes$ is the
pushforward of the full array law under the map
\[
  \text{full array}\longmapsto (A_m^X,A_m^Y).
\]
This map averages $m$ disjoint $X$-coordinates and $m$ disjoint
$Y$-coordinates, and has Lipschitz constant $m^{-1/2}$ for the coordinate
Euclidean metrics.  Wasserstein distance decreases under Lipschitz pushforward
\cite[Ch.~6]{villani2009optimal},
hence
\[
  W_1(Q_m^\Delta,Q_m^\otimes)
  \le
  m^{-1/2}
  W_1\bigl(\Phi_{2,m}(\lambda_\Delta),\Phi_{2,m}(\mu\otimes\mu)\bigr),
\]
which gives the displayed lower bound for the full array laws.
\end{proof}

\begin{proposition}[A spectral Wasserstein rate for reversible Markov shifts]
\label{prop:markov-spectral-rate}
Let $E$ be a finite set with at least two points, let $P$ be an irreducible
aperiodic transition matrix on $E$, and let $\pi$ be its stationary law.  Assume
that $P$ is reversible on $L^2(\pi)$, and write
\[
  \theta=\max\{|\alpha|:\alpha\in\operatorname{Spec}(P),\ \alpha\ne1\}<1 .
\]
Let $\mu_P$ be the stationary one-sided Markov measure on
$X=E^{\mathbb N_0}$, and let $T$ be the left shift.  Fix
$0<\beta<1$ and $\eta>0$ such that
\[
  \tau:=\eta\sum_{r\ge1}\beta^r=\frac{\eta\beta}{1-\beta}<1,
\]
and equip $X$ with the compatible metric
\[
  d_{\beta,\eta}(x,x')
  =
  \mathbf 1_{\{x_0\ne x'_0\}}
  +\eta\sum_{r\ge1}\beta^r\mathbf 1_{\{x_r\ne x'_r\}} .
\]
For $m\ge1$ define the normalized two-particle mismatch observable
\[
  B_m(x_1,x_2)
  =
  \frac1m\sum_{t=0}^{m-1}\mathbf 1_{\{x_{1,t}\ne x_{2,t}\}} .
\]
Let $\lambda_\Delta=(x\mapsto(x,x))_\#\mu_P$ be the diagonal joining of
$(X,d_{\beta,\eta},\mu_P,T)$ with itself.  Let $Q_m^\Delta$ and $Q_m^\otimes$
be the laws on $[0,1]^2$ of
\[
  (B_m(x_1,x_2),B_m(y_1,y_2))
\]
under $\lambda_\Delta^{\otimes2}$ and $(\mu_P\otimes\mu_P)^{\otimes2}$,
respectively.  Put
\[
  h(i,j)=\mathbf 1_{\{i\ne j\}},\qquad
  v_h=\operatorname{Var}_{\pi\otimes\pi}(h),
\]
and let
\[
  \sigma_h^2
  =
  v_h+2\sum_{t=1}^\infty
  \operatorname{Cov}_{\pi\otimes\pi}
  \bigl(h(Z_0),h(Z_t)\bigr),
\]
where $(Z_t)$ is the stationary Markov chain on $E\times E$ with transition
$P\otimes P$.  Then $\sigma_h^2>0$ and, for every $m\ge1$,
\[
  W_1(Q_m^\Delta,Q_m^\otimes)
  \le
  \left(
    \frac{2v_h(1+\theta)}{m(1-\theta)}
  \right)^{1/2}.
\]
Moreover,
\[
  \liminf_{m\to\infty}
  \sqrt m\, W_1(Q_m^\Delta,Q_m^\otimes)
  \ge
  \frac{\sqrt2\,\sigma_h}{\sqrt\pi}.
\]
In particular, for all sufficiently large $m$,
\[
  W_1(Q_m^\Delta,Q_m^\otimes)
  \ge
  \frac{\sigma_h}{\sqrt{2\pi m}} .
\]
Thus this normalized Markov distance-array projection has the sharp order
$m^{-1/2}$.

Finally, $Q_m^\Delta$ and $Q_m^\otimes$ are Lipschitz projections of the full
unanchored laws $\Phi_{2,m}(\lambda_\Delta)$ and
$\Phi_{2,m}(\mu_P\otimes\mu_P)$.  More precisely,
\[
  W_1(Q_m^\Delta,Q_m^\otimes)
  \le
  \frac{1}{(1-\tau)\sqrt m}\,
  W_1\bigl(\Phi_{2,m}(\lambda_\Delta),
           \Phi_{2,m}(\mu_P\otimes\mu_P)\bigr).
\]
Consequently the full raw two-particle array laws remain separated along a
subsequence; indeed
\[
  \liminf_{m\to\infty}
  W_1\bigl(\Phi_{2,m}(\lambda_\Delta),
           \Phi_{2,m}(\mu_P\otimes\mu_P)\bigr)
  \ge
  \frac{\sqrt2(1-\tau)\sigma_h}{\sqrt\pi}.
\]
\end{proposition}

\begin{proof}
Let $(X_t)$ and $(X'_t)$ be independent stationary Markov chains with
transition matrix $P$ and common stationary law $\pi$.  Then
$Z_t=(X_t,X'_t)$ is a stationary reversible Markov chain on $E\times E$ with
transition matrix $R=P\otimes P$ and stationary law $\Pi=\pi\otimes\pi$.
Since $P$ is irreducible and aperiodic, the product chain $R=P\otimes P$ is
also irreducible and aperiodic.
The nontrivial eigenvalues of $R$ are products of eigenvalues of $P$, with at
least one factor different from $1$.  Hence the $L^2_0(\Pi)$ spectral radius
of $R$ is at most $\theta$.

Set $\bar h=h-\int h\,\dd\Pi$, and write
\[
  \gamma_t=\operatorname{Cov}_\Pi(h(Z_0),h(Z_t))
  =
  \langle \bar h,R^t\bar h\rangle_{L^2(\Pi)} .
\]
The spectral bound gives $|\gamma_t|\le \theta^t v_h$.  Since
\[
  \operatorname{Var}(B_m)
  =
  \frac{1}{m^2}
  \left(
    mv_h+2\sum_{t=1}^{m-1}(m-t)\gamma_t
  \right),
\]
we obtain
\[
  \operatorname{Var}(B_m)
  \le
  \frac{v_h}{m}
  \left(1+2\sum_{t=1}^{m-1}\theta^t\right)
  \le
  \frac{v_h(1+\theta)}{m(1-\theta)}.
\]

Under $\lambda_\Delta^{\otimes2}$ the two coordinates in $Q_m^\Delta$ are the
same random variable $B_m$, while under the product joining they are
independent copies $B_m,B'_m$.  Coupling these laws by using the same first
coordinate gives
\[
  W_1(Q_m^\Delta,Q_m^\otimes)
  \le
  \mathbb E|B_m-B'_m|
  \le
  \sqrt{2\operatorname{Var}(B_m)},
\]
which proves the upper bound.

It remains to prove the lower rate.  Because $P$ is finite-state, irreducible,
and aperiodic, the standard central limit theorem for additive functionals of
finite Markov chains applies to $\bar h$; see, for example,
\cite{meyn2009markov,levin2017markov}.  Thus
\[
  \sqrt m\left(B_m-\int h\,\dd\Pi\right)
  \Rightarrow
  N(0,\sigma_h^2).
\]
The spectral representation
\[
  \sigma_h^2
  =
  \sum_{\rho\in\operatorname{Spec}(R|_{L^2_0(\Pi)})}
  \frac{1+\rho}{1-\rho}\,\|E_\rho\bar h\|_2^2
\]
also shows $\sigma_h^2>0$: since $R$ is irreducible and aperiodic, its
nontrivial eigenvalues lie in $(-1,1)$, so all coefficients
$(1+\rho)/(1-\rho)$ are positive; moreover $\bar h\ne0$ because $|E|\ge2$ and
$\pi$ has full support.  Applying the same central limit theorem to an
independent copy gives
\[
  \sqrt m\,(B_m-B'_m)\Rightarrow G-G' .
\]
The variance bound above gives a uniform $L^2$ bound for
$\sqrt m\,(B_m-B'_m)$, and hence uniform integrability of its absolute value.
Therefore
\[
  \sqrt m\,\mathbb E|B_m-B'_m|
  \longrightarrow
  \mathbb E|G-G'|
  =
  \frac{2\sigma_h}{\sqrt\pi},
\]
where $G,G'$ are independent $N(0,\sigma_h^2)$ variables.  The function
\[
  \varphi(a,b)=\frac{|a-b|}{\sqrt2},\qquad 0\le a,b\le1,
\]
is $1$-Lipschitz for the Euclidean metric on $[0,1]^2$ and vanishes on the
diagonal.  Kantorovich--Rubinstein duality for $W_1$
\cite[Ch.~1]{villani2009optimal} therefore yields
\[
  W_1(Q_m^\Delta,Q_m^\otimes)
  \ge
  \frac{1}{\sqrt2}\mathbb E|B_m-B'_m|.
\]
Taking the lower limit proves the asserted asymptotic lower bound, and the
eventual bound follows immediately.

It remains only to connect this normalized observable to the distance-array
coordinates.  Choose a function $\chi:[0,1+\tau]\to[0,1]$ such that
\[
  \chi=0\text{ on }[0,\tau],\qquad
  \chi=1\text{ on }[1,1+\tau],\qquad
  \operatorname{Lip}(\chi)\le\frac1{1-\tau}.
\]
For every $x,x'\in X$,
\[
  \mathbf 1_{\{x_0\ne x'_0\}}=\chi(d_{\beta,\eta}(x,x')).
\]
Hence $B_m(x_1,x_2)$ is obtained from the $m$ array coordinates
$d_{\beta,\eta}(T^t x_1,T^t x_2)$, $0\le t<m$, by applying $\chi$ and averaging.
The map from the full two-particle array to $(B_m^X,B_m^Y)$ depends on the
$m$ coordinates
\[
  d_{\beta,\eta}(T^t x_1,T^t x_2),\qquad 0\le t<m,
\]
and the disjoint $m$ coordinates
\[
  d_{\beta,\eta}(T^t y_1,T^t y_2),\qquad 0\le t<m .
\]
The two output coordinates have orthogonal gradients.  More explicitly, at
points where $\chi$ is differentiable the Jacobian of this projection has
diagonal Gram matrix whose two diagonal entries are bounded by
\[
  \frac{1}{m^2}\sum_{t=0}^{m-1}\operatorname{Lip}(\chi)^2
  \le
  \frac{1}{(1-\tau)^2m}.
\]
Since $\chi$ may be chosen Lipschitz and is differentiable almost everywhere,
the same bound follows for difference quotients by approximation along line
segments.  Thus the projection has Lipschitz constant at most
$1/((1-\tau)\sqrt m)$ for the coordinate Euclidean metrics.  Pushing
Wasserstein distance forward by this map, using the Lipschitz contraction
property of Wasserstein distance \cite[Ch.~6]{villani2009optimal}, gives the displayed projection
inequality, and combining it with the asymptotic lower bound gives the final
separation estimate for the raw arrays.
\end{proof}

\section{Examples and consequences}

\begin{proposition}[Graph-joining quotient class for compact rotations]
\label{prop:rotation-graph-class}
In the setting of Proposition~\ref{prop:single-fails}, the graph joinings
\[
  \lambda_h=(x\mapsto(x,x+h))_\#m_G,\qquad h\in G,
\]
all lie in one coordinate-isomorphism class and have the same unanchored
distance-array laws:
\[
  \Phi_{n,m}(\lambda_h)=\Phi_{n,m}(\lambda_{h'})
  \qquad\forall h,h'\in G,\ \forall n,m\ge1.
\]
Moreover, if the random variable $d(X_1,X_2)$ is non-degenerate for
independent Haar points $X_1,X_2$, then this graph-joining class is separated
from the product joining by the two-particle array $\Phi_{2,1}$.
\end{proposition}

\begin{proof}
Let $\tau_t(x)=x+t$.  Translation invariance gives
$d(\tau_t x,\tau_t x')=d(x,x')$, and $\tau_t$ commutes with $T_g$ because
$G$ is abelian.  Therefore $\tau_t\in\Aut_{\rm md}(G,d,m_G,T_g)$ and
\[
  (\Id\times\tau_{h'-h})_\#\lambda_h=\lambda_{h'}.
\]
This proves that all graph joinings lie in one coordinate-isomorphism class.

For the array laws, if $(x_i,y_i)_{i=1}^n$ is sampled from
$\lambda_h^{\otimes n}$, then $y_i=x_i+h$ for every $i$.  Hence, for all
$a,b$ and all $i,j$,
\[
  d(T_g^a y_i,T_g^b y_j)
  =
  d(x_i+h+ag,x_j+h+bg)
  =
  d(T_g^a x_i,T_g^b x_j).
\]
Thus the $Y$-array is identical to the $X$-array and the law is independent of
$h$.

For $n=2,m=1$, the graph joining gives the pair
\[
  \bigl(d(x_1,x_2),d(y_1,y_2)\bigr)=(R,R)
\]
with $R=d(X_1,X_2)$.  Under the product joining the corresponding pair is
$(R,R')$, where $R'$ is an independent copy of $R$.  If $R$ is non-degenerate,
then
\[
  \mathbb E(R-R)^2=0,
  \qquad
  \mathbb E(R-R')^2=2\operatorname{Var}(R)>0.
\]
Thus the laws of $(R,R)$ and $(R,R')$ are different, and therefore
$\Phi_{2,1}(\lambda_h)\ne\Phi_{2,1}(m_G\otimes m_G)$.
\end{proof}

\begin{proposition}[A low-order Bernoulli example]
\label{prop:bernoulli-low-order}
Let $X=\{0,1\}^{\mathbb Z}$, let $\mu=(\frac12,\frac12)^{\mathbb Z}$, and let
$T$ be the two-sided Bernoulli shift.  Put
\[
  d(x,x')=\sum_{k\in\mathbb Z}2^{-|k|-2}|x_k-x'_k|.
\]
Consider the two identical systems $\Xbf=\Ybf=(X,d,\mu,T)$, and choose dense
anchor families with first anchor equal to the all-zero sequence $0^\infty$ in
both coordinates.  Then for every $p\ge1$,
\[
  \widetilde{\operatorname{Dep}}_{1,1,1,p}(\Xbf,\Ybf)>0 .
\]
\end{proposition}

\begin{proof}
Let $\lambda_\Delta=(x\mapsto(x,x))_\#\mu$ be the diagonal joining.  The
anchored observable at order $(n,m,R)=(1,1,1)$ contains the two real variables
\[
  D_X=d(x,0^\infty),\qquad D_Y=d(y,0^\infty).
\]
Under $\lambda_\Delta$ one has $D_X=D_Y$ almost surely.  Under the product
joining $\mu\otimes\mu$, the variables $D_X$ and $D_Y$ are independent copies.
The variable
\[
  D_X=\sum_{k\in\mathbb Z}2^{-|k|-2}x_k
\]
is not almost surely constant, since its variance is
\[
  \frac14\sum_{k\in\mathbb Z}2^{-2|k|-4}>0.
\]
Consequently,
\[
  \mathbb E_{\lambda_\Delta}(D_X-D_Y)^2=0,
  \qquad
  \mathbb E_{\mu\otimes\mu}(D_X-D_Y)^2=2\operatorname{Var}(D_X)>0.
\]
Hence the law of $(D_X,D_Y)$ under $\lambda_\Delta$ differs from its law under
$\mu\otimes\mu$.  Since $(D_X,D_Y)$ is a coordinate projection of the anchored
order $(1,1,1)$ array, the full anchored array law
$\widetilde\Phi_{1,1,1}(\lambda_\Delta)$ therefore differs from
$\widetilde\Phi_{1,1,1}(\mu\otimes\mu)$.  Since $W_p$ is a metric on
probability measures on the compact finite array cube and the dependence
coefficient is a supremum over joinings, the displayed coefficient is strictly
positive.
\end{proof}

\begin{proposition}[Common factors create nonzero anchored dependence]
\label{prop:common-factor}
Suppose $\Xbf$ and $\Ybf$ have a common factor
$\Zbf=(Z,\eta,R)$: there are factor maps $\pi_X:X\to Z$ and $\pi_Y:Y\to Z$
with $(\pi_X)_\#\mu=(\pi_Y)_\#\nu=\eta$,
$\pi_X\circ T=R\circ\pi_X$, and $\pi_Y\circ S=R\circ\pi_Y$.  Assume that the
factor is seen nontrivially from both systems, in the sense that
$\pi_X^{-1}\mathcal B_Z$ is nontrivial modulo $\mu$ and
$\pi_Y^{-1}\mathcal B_Z$ is nontrivial modulo $\nu$.  Then the systems are not
disjoint.  Consequently, for any dense support anchors and any $p\ge1$, at least one
anchored coefficient
$\widetilde{\operatorname{Dep}}_{n,m,R,p}$ is strictly positive.
\end{proposition}

\begin{proof}
The relatively independent joining over $Z$,
\[
  \lambda=\int_Z \mu_z\otimes\nu_z\,\dd\eta(z),
\]
where $\mu=\int\mu_z\,\dd\eta(z)$ and $\nu=\int\nu_z\,\dd\eta(z)$ are the
disintegrations over the common factor, is invariant under $T\times S$ and has
marginals $\mu,\nu$; this is the standard relatively independent joining over
a factor \cite[Ch.~6]{glasner2003ergodic}.  Since the factor is nontrivial, choose
$B\in\mathcal B_Z$ with $0<\eta(B)<1$.  Set
$A_X=\pi_X^{-1}B$ and $A_Y=\pi_Y^{-1}B$.  Under the relatively independent
joining,
\[
  \lambda(A_X\times A_Y)=\eta(B),
\]
whereas
\[
  (\mu\otimes\nu)(A_X\times A_Y)=\eta(B)^2.
\]
Thus $\lambda\ne\mu\otimes\nu$.  The systems are not disjoint.  If all anchored
coefficients were zero, Theorem~\ref{thm:zero} would imply disjointness, a
contradiction.
\end{proof}

\begin{corollary}[Kronecker factor obstruction]
\label{cor:kronecker-obstruction}
Let $\Xbf=(X,d_X,\mu,T)$ be a compact metric model of an ergodic system.  If
the underlying system has a nontrivial Kronecker factor, then there exists an
ergodic compact abelian rotation $\Ybf=(Y,d_Y,\nu,S)$ such that, for any dense
support anchor families and any $p\ge1$,
\[
  \widetilde{\operatorname{Dep}}_{n,m,R,p}(\Xbf,\Ybf)>0
\]
for at least one finite triple $(n,m,R)$.
\end{corollary}

\begin{proof}
A nontrivial Kronecker factor is measurably isomorphic to an ergodic rotation
on a compact abelian group \cite[Ch.~2]{petersen1983ergodic}.  Choose such a compact abelian group rotation
model $\Ybf=(Y,d_Y,\nu,S)$ for the Kronecker factor, and compose the original
measure-theoretic factor map from $X$ with this model isomorphism to obtain
$\pi:X\to Y$ modulo null sets.  Then $\Xbf$ and $\Ybf$ have the common factor
$\Ybf$, seen through $\pi$ on the first coordinate and through the identity map
on the second.  The factor $\sigma$-algebra is nontrivial by assumption.
Proposition~\ref{prop:common-factor} therefore gives non-disjointness and
hence, by Theorem~\ref{thm:zero}, at least one anchored finite-pattern
dependence coefficient must be strictly positive.
\end{proof}

\begin{corollary}[Weak mixing through distance-array projections]
Let $\Xbf$ be a compact metric model of an ergodic probability-preserving
system, and fix dense support anchors on $\Xbf$.  If the underlying system of
$\Xbf$ is weakly mixing, then for every ergodic compact metric Kronecker model
$\Ybf$, for any fixed dense support anchors on $\Ybf$, and for every fixed
$p\ge1$,
\[
  \widetilde{\operatorname{Dep}}_{n,m,R,p}(\Xbf,\Ybf)=0
  \qquad\forall n,m,R\ge1.
\]
Conversely, if for some fixed $p\ge1$ this all-order vanishing holds for every
ergodic compact metric Kronecker model $\Ybf$ and for any fixed dense support
anchors on $\Ybf$, then the underlying system of $\Xbf$ is weakly mixing.
\end{corollary}

\begin{proof}
An ergodic system is weakly mixing if and only if it is disjoint from every
ergodic Kronecker system.  This is the classical Furstenberg disjointness
characterization of weak mixing
\cite{furstenberg1981recurrence,glasner2003ergodic,petersen1983ergodic}.
If the underlying system of $\Xbf$ is weakly mixing, it is disjoint from every
ergodic Kronecker system, and Theorem~\ref{thm:zero} gives the displayed
vanishing for every ergodic compact metric Kronecker model and for any fixed
dense support anchors on that model.  Conversely, if the displayed
vanishing holds for every ergodic compact metric Kronecker model and for any
fixed dense support anchors on that model, Theorem~\ref{thm:zero} implies
disjointness from every such Kronecker system, hence weak mixing by the same
classical characterization.
\end{proof}

\section{Finite observable certificates}

The anchored criterion has a useful consequence that is not naturally visible
from the classical definition of disjointness.  A non-product joining is usually
witnessed by abstract Borel sets in $X\times Y$.  The next result says that, in
compact metric models, such a witness can always be replaced by a finite
continuous statistic of anchored orbit-distance arrays.  Moreover this can be
done uniformly on compact families of alternative joinings.

\begin{theorem}[Finite distance-array certificates]
\label{thm:finite-certificate}
Assume $T$ and $S$ are continuous, and fix dense support anchors.  Let
$\lambda_0=\mu\otimes\nu$.
\begin{enumerate}
\item If $\lambda\in\J(T,S)$ and $\lambda\ne\lambda_0$, then there are
integers $n,m,R\ge1$ and a $1$-Lipschitz continuous function $\psi$ on the
finite anchored array cube such that
\[
  \int \psi\,\dd\widetilde\Phi_{n,m,R}(\lambda)
  \ne
  \int \psi\,\dd\widetilde\Phi_{n,m,R}(\lambda_0).
\]
\item Let $\mathcal U$ be any weak neighborhood of $\lambda_0$ in $\J(T,S)$.
Then there exist finitely many triples $(n_\ell,m_\ell,R_\ell)$, finitely many
$1$-Lipschitz continuous functions $\psi_\ell$ on the corresponding anchored
array cubes, and a number $\varepsilon>0$ such that every joining
$\lambda\in\J(T,S)$ satisfying
\[
  \left|
  \int\psi_\ell\,\dd\widetilde\Phi_{n_\ell,m_\ell,R_\ell}(\lambda)
  -
  \int\psi_\ell\,\dd\widetilde\Phi_{n_\ell,m_\ell,R_\ell}(\lambda_0)
  \right|
  <\varepsilon
  \qquad\text{for every }\ell
\]
belongs to $\mathcal U$.
\end{enumerate}
\end{theorem}

\begin{proof}
For the first assertion, suppose that no such finite statistic exists.  Then
for every $n,m,R$ the probability measures
\[
  \widetilde\Phi_{n,m,R}(\lambda),
  \qquad
  \widetilde\Phi_{n,m,R}(\lambda_0)
\]
have the same integral against every $1$-Lipschitz continuous function on the
finite array cube.  The cube is compact metric, and the $1$-Wasserstein
distance satisfies the Kantorovich--Rubinstein dual formula
\cite[Ch.~1]{villani2009optimal}.  Hence equality
of integrals against all $1$-Lipschitz test functions is equivalent to
$W_1=0$, and therefore to equality of probability measures on the compact
finite array cube.  Thus the two finite array laws are equal for every
$n,m,R$.  By the anchored reconstruction theorem, Theorem~\ref{thm:anchored},
this implies $\lambda=\lambda_0$, a contradiction.  Thus a finite certificate
exists.

For the second assertion, set
\[
  K=\J(T,S)\setminus\mathcal U .
\]
Since $T$ and $S$ are continuous, $\J(T,S)$ is weakly compact by
Proposition~\ref{prop:compact}, and $K$ is compact.  If $K=\emptyset$, the
claim is vacuous after choosing any one finite statistic and any
$\varepsilon>0$.  Assume $K\ne\emptyset$.  For each
$\lambda\in K$, the first part gives a triple $(n_\lambda,m_\lambda,R_\lambda)$
and a $1$-Lipschitz function $\psi_\lambda$ with a nonzero gap
\[
  \Delta_\lambda=
  \left|
  \int\psi_\lambda\,\dd\widetilde\Phi_{n_\lambda,m_\lambda,R_\lambda}(\lambda)
  -
  \int\psi_\lambda\,\dd\widetilde\Phi_{n_\lambda,m_\lambda,R_\lambda}(\lambda_0)
  \right|>0 .
\]
Because the anchored array map is continuous in $\lambda$ and $\psi_\lambda$
is continuous, the set
\[
  U_\lambda=
  \left\{
  \eta\in\J(T,S):
  \left|
  \int\psi_\lambda\,\dd\widetilde\Phi_{n_\lambda,m_\lambda,R_\lambda}(\eta)
  -
  \int\psi_\lambda\,\dd\widetilde\Phi_{n_\lambda,m_\lambda,R_\lambda}(\lambda_0)
  \right|
  >
  \frac{\Delta_\lambda}{2}
  \right\}
\]
is an open neighborhood of $\lambda$.  The family $\{U_\lambda:\lambda\in K\}$
covers $K$, so by compactness there are
$\lambda_1,\ldots,\lambda_L\in K$ such that
$K\subset\bigcup_{\ell=1}^L U_{\lambda_\ell}$.  Put
\[
  \varepsilon=\frac12\min_{1\le\ell\le L}\Delta_{\lambda_\ell}>0
\]
and use the corresponding triples and test functions.  If a joining $\eta$
satisfies all displayed inequalities with this $\varepsilon$, then
$\eta\notin U_{\lambda_\ell}$ for every $\ell$, hence $\eta\notin K$ and
therefore $\eta\in\mathcal U$.
\end{proof}

\begin{corollary}[Finite certificates for non-disjointness]
\label{cor:finite-certificate-nondisjoint}
Assume $T$ and $S$ are continuous and dense support anchors are fixed.  If
$\Xbf$ and $\Ybf$ are not disjoint, then there are $n,m,R\ge1$ and a
$1$-Lipschitz continuous anchored distance-array statistic $\psi$ such that
\[
  \sup_{\lambda\in\J(T,S)}
  \left|
  \int \psi\,\dd\widetilde\Phi_{n,m,R}(\lambda)
  -
  \int \psi\,\dd\widetilde\Phi_{n,m,R}(\mu\otimes\nu)
  \right|>0 .
\]
In particular,
\[
  \widetilde{\operatorname{Dep}}_{n,m,R,1}(\Xbf,\Ybf)>0
\]
for some finite triple $(n,m,R)$.
\end{corollary}

\begin{proof}
Non-disjointness gives a joining $\lambda\ne\mu\otimes\nu$.  Apply the first
part of Theorem~\ref{thm:finite-certificate} to this joining.  The final
claim follows from the Kantorovich--Rubinstein dual representation of $W_1$
\cite[Ch.~1]{villani2009optimal}
and the definition of the anchored dependence coefficient.
\end{proof}

\section{Discussion}

The results separate three levels of information.  Single-orbit distance
patterns are too poor: in homogeneous examples they may be deterministic.
Unanchored multi-particle arrays are intrinsic and reconstruct the joining as a
marked colored orbit-distance kernel space, with the unavoidable ambiguity of
twin-free quotients and metric-dynamical symmetries.  Anchored arrays remove
that ambiguity and give a fixed-model criterion equivalent to ordinary
Furstenberg disjointness.

This distinction is important.  The intrinsic theory is not a disguised
coordinate reconstruction theorem; it is a reconstruction theorem at the level
that finite unanchored metric data can naturally see.  The anchored theory, by
contrast, deliberately adds external names through a dense family of metric
anchors, and therefore recovers literal equality of joinings on a fixed compact
model.  The two viewpoints are complementary: one is invariant under the
unavoidable symmetries of metric observation, while the other is faithful to a
chosen model.

The Wasserstein coefficients should also be interpreted at this level.  A
finite coefficient depends on the chosen compact metric model and on the finite
observation scale.  Its all-order vanishing, however, is a measure-theoretic
statement because it is equivalent to Furstenberg disjointness in the anchored
hierarchy.  The finite-certificate theorem adds a compactness principle: once a
weak neighborhood of the product joining is prescribed, finitely many anchored
orbit-distance statistics suffice to force every joining into that
neighborhood.

Several questions remain open.  First, it would be natural to extend the
criterion beyond compact metric models, for instance to locally compact or
Polish models under moment assumptions on the distance arrays.  Second, the
quantitative hierarchy may carry more structure than its zero set: for
particular classes of systems one can ask whether rates of decay or separation
reflect spectral gaps, rates of mixing, or the presence of compact factors.
Third, the intrinsic quotient reconstruction suggests a finer comparison of
joining spaces themselves, not only the dichotomy between product and
non-product joinings.  These directions point toward a metric reconstruction
view of joining theory in which finite orbit-distance patterns are not auxiliary
statistics, but finite shadows of the joining space.

\section*{Acknowledgements}

The author gratefully acknowledges the support of Jilin University and
Zhongguancun Academy.  This work is supported by the Zhongguancun Academy,
Grant No.~C20250201.

\section*{Declarations}

\subsection*{Conflict of interest}
The author declares no competing interests.

\subsection*{Data availability}
No datasets were generated or analysed during the current study.

\bibliography{references}

\end{document}